\documentclass[brochure,english,12pt,reqno]{bourbaki}
\usepackage[matrix,arrow]{xy}
\usepackage{amssymb,amsfonts,amsmath,footnote,hyperref}
\usepackage[francais]{babel}
\date{Janvier 2011}
\bbkannee{63\`eme ann\'ee, 2010-2011}
\bbknumero{1033}
\title{Sparse quadratic forms and their geometric applications}
\subtitle{after Batson, Spielman and Srivastava}
\author{Assaf NAOR}
\address{New York University\\
Courant Institute\\
251 Mercer street\\
New York, NY 10012  --  USA}
\email{naor@cims.nyu.edu}
\thanks{Supported in part by NSF grant CCF-0635078, BSF
grant 2006009, and the Packard Foundation.}

\renewcommand{\le}{\leqslant}
\renewcommand{\ge}{\geqslant}
\renewcommand{\leq}{\leqslant}

\renewcommand{\setminus}{\smallsetminus}

\renewcommand{\1}{\mathbf{1}}
\newcommand{\e}{\varepsilon}
\newcommand{\R}{\mathbb{R}}

\newcommand{\N}{\mathbb{N}}

\newcommand{\supp}{\mathrm{supp}}
\newcommand{\trace}{\mathrm{tr}}
\newcommand{\HS}{\mathrm{HS}}
\newcommand{\Ker}{\mathrm{Ker}}
\newcommand{\tr}{\mathrm{tr}}

\begin{document}
\maketitle

\section{Introduction}

In what follows all matrices are assumed to have real entries, and square matrices are always assumed to be symmetric unless stated otherwise. The support of a $k\times n$ matrix $A=(a_{ij})$ will be denoted below by
$$
\supp(A)=\big\{(i,j)\in \{1,\ldots,k\}\times \{1,\ldots,n\}:\ a_{ij}\neq 0\big\}.
$$
If $A$ is an $n\times n$  matrix, we denote the decreasing rearrangement of its eigenvalues by
$$\lambda_1(A)\ge \lambda_2(A)\ge \cdots\ge \lambda_n(A).
$$
$\R^n$ will always be assumed to be equipped with the standard scalar product $\langle\cdot,\cdot\rangle$. Given a vector $v\in \R^n$ and $i\in \{1,\ldots,n\}$, we denote by $v_i$ the $i$th coordinate of $v$. Thus for $u,v\in \R^n$ we have $\langle u,v\rangle =\sum_{i=1}^n u_iv_i$.

Our goal here is to describe the following theorem of Batson, Spielman and Srivastava~\cite{BSS}, and to explain some of its recently discovered geometric applications. We expect that there exist many more applications of this fundamental fact in matrix theory.
 \begin{theo}\label{thm:BSS}
For every $\e\in (0,1)$ there exists $c(\e)=O(1/\e^2)$ with the following properties.  Let $G=(g_{ij})$ be an $n\times n$ matrix with nonnegative entries. Then there exists an $n\times n$ matrix $H=(h_{ij})$ with nonnegative entries that satisfies the following conditions:
\begin{enumerate}
\item $\supp(H)\subseteq \supp(G)$.
\item The cardinality of the support of $H$ satisfies $|\supp(H)|\le c(\e) n$.
\item For every $x\in \R^n$ we have
\begin{equation}\label{eq:sparsification}
\sum_{i=1}^n\sum_{j=1}^n g_{ij}(x_i-x_j)^2\le \sum_{i=1}^n\sum_{j=1}^n h_{ij}(x_i-x_j)^2\le (1+\e)\sum_{i=1}^n\sum_{j=1}^n g_{ij}(x_i-x_j)^2.
\end{equation}
\end{enumerate}
 \end{theo}
 The second assertion of Theorem~\ref{thm:BSS} is that the matrix $H$ is {\em sparse}, yet due to the third assertion of Theorem~\ref{thm:BSS} the quadratic form $\sum_{i=1}^n\sum_{j=1}^n h_{ij}(x_i-x_j)^2$ is nevertheless a good approximation of the quadratic form $\sum_{i=1}^n\sum_{j=1}^n g_{ij}(x_i-x_j)^2$. For this reason Theorem~\ref{thm:BSS} is called in the literature a {\em sparsification theorem}.

 The bound on $|\supp(H)|$ obtained in~\cite{BSS} is
\begin{equation}\label{eq:c(eps)}
|\supp(H)|\le 2\left\lceil\frac{(\sqrt{1+\e}+1)^4}{\e^2}n\right\rceil.
\end{equation}
Thus $c(\e)\le 32/\e^2+O(1/\e)$. There is no reason to expect that~\eqref{eq:c(eps)} is best possible, but a simple argument~\cite[Section~4]{BSS} shows that necessarily $c(\e)\ge 8/\e^2$.


\subsection{Historical discussion}


The sparsification problem that is solved (up to constant factors) by Theorem~\ref{thm:BSS} has been studied for some time in the theoretical computer science literature. The motivations for these investigations were algorithmic, and therefore there was emphasis on constructing the matrix $H$ quickly. We will focus here on geometric applications of Theorem~\ref{thm:BSS} for which the existential statement suffices, but we do wish to state that~\cite{BSS} shows that $H$ can be constructed in time $O(n^3|\supp(G)|/\e^2)=O(n^5/\e^2)$. For certain algorithmic applications this running time is too slow, and the literature contains works that yield weaker asymptotic bounds on $|\supp(H)|$ but have a faster construction time. While such tradeoffs are important variants of Theorem~\ref{thm:BSS}, they are not directly relevant to our discussion and we will not explain them here. For the applications described below, even a weaker bound of, say, $|\supp(H)|\le c(\e)n\log n$ is insufficient.

Bencz\'ur and Karger~\cite{BK} were the first to study the sparsification problem. They proved the existence of a matrix $H$ with $|\supp(H)|\le c(\e)n\log n$, that satisfies the conclusion~\eqref{eq:sparsification} only for {\em Boolean} vectors $x\in \{0,1\}^n$. In their series of works on fast solvers for certain linear systems~\cite{ST1,ST2,ST3,ST4}, Spielman and Teng studied the sparsification problem as stated in Theorem~\ref{thm:BSS}, i.e., with the conclusion~\eqref{eq:sparsification} holding for {\em every} $x\in \R^n$. Specifically, in~\cite{ST4}, Spielman and Teng proved Theorem~\ref{thm:BSS} with the weaker estimate $|\supp(H)|= O\left(n(\log n)^7/\e^2\right)$.  Spielman and Srivastava~\cite{SS1} improved this estimate on the size of the support of $H$ to $|\supp(H)|=O(n(\log n)/\e^2)$. As we stated above, Theorem~\ref{thm:BSS}, which answers positively a conjecture of Spielman-Srivastava~\cite{SS1}, is due to Batson-Spielman-Srivastava~\cite{BSS}, who proved this sharp result via a new deterministic iterative technique (unlike the previous probabilistic arguments) that we will describe below. This beautiful new approach does not only yield an asymptotically sharp bound on $|\supp(H)|$: it gives for the first time a deterministic algorithm for constructing $H$ (unlike the previous randomized algorithms), and it also gives additional results that will be described later. We refer to Srivastava's dissertation~\cite{Sr2} for a very nice and more complete exposition of these ideas. See also the work of Kolla-Makarychev-Saberi-Teng~\cite{KMST} for  additional results along these lines.

\subsection{Combinatorial interpretation}

Suppose that $G$ is the adjacency matrix of the complete graph, i.e., the diagonal entries of $G$ vanish and $g_{ij}=1$ if $i\neq j$. Assume also that the matrix $H$ of Theorem~\ref{thm:BSS} happens to be a multiple of the adjacency matrix of a $d$-regular graph $\Gamma=(\{1,\ldots,n\},E)$, i.e., for some $\gamma>0$ and all $i,j\in \{1,\ldots,n\}$ we have $h_{ij}=\gamma$ if $\{i,j\}\in E$ and $h_{ij}=0$ otherwise. Thus $|\supp(H)|=dn$. By expanding the squares in~\eqref{eq:sparsification} and some straightforward linear algebra, we see that~\eqref{eq:sparsification} is equivalent to the bound $(\lambda_1(H)-\lambda_n(H))/(\lambda_1(H)-\lambda_2(H))\le 1+\e$. Thus if $\e$ is small then the graph $\Gamma$ is a good expander (see~\cite{HLW} for background on this topic). The Alon-Boppana bound~\cite{Ni} implies that $H$ satisfies
$(\lambda_1(H)-\lambda_n(H))/(\lambda_1(H)-\lambda_2(H))\ge 1+4(1-o(1))\sqrt{d}$ as $n,d\to \infty$. This lower bound can be asymptotically attained since if $\Gamma$ is a Ramanujan graph of Lubotzky-Phillips-Sarnak~\cite{LPS} then $\lambda_1(H)/\gamma,\lambda_n(H)/\gamma\in \left[-2\sqrt{d-1},2\sqrt{d-1}\right]$. Writing $1+\e=\left(d+2\sqrt{d-1}\right)/\left(d-2\sqrt{d-1}\right)=1+4(1+o(1))/\sqrt{d}$, we see that the existence of Ramanujan graphs means that (in this special case of the complete graph) there exists a matrix $H$ satisfying~\eqref{eq:sparsification} with $|\supp(H)|=dn=16n(1+o(1))/\e^2$. The bound on $|\supp(H)|$ in~\eqref{eq:c(eps)} shows that Thereom~\ref{thm:BSS} achieves the optimal Ramanujan bound up to a factor of $2$. For this reason Batson-Spielman-Srivastava call the matrices produced by Theorem~\ref{thm:BSS} ``twice-Ramanujan sparsifiers".
Of course, this analogy is incomplete since while the matrix $H$ is sparse, it need not be a multiple of the adjacency matrix of a graph, but rather an adjacency matrix of a weighted graph. Moreover, this graph has bounded average degree, rather than being a regular graph of bounded degree. Such weighted sparse (though non-regular) graphs  still have useful pseudorandom properties (see~\cite[Lemma~4.1]{BSS}). Theorem~\ref{thm:BSS} can be therefore viewed as a new deterministic construction of ``expander-like" weighted graphs, with very good spectral gap. Moreover, it extends the notion of expander graphs since one can start with an arbitrary matrix $G$ before applying the sparsification procedure, with the quality of the resulting expander (measured in terms of absolute spectral gap) being essentially the same as the quality of $G$ as an expander.

\subsection{ Structure of this paper.} In Section~\ref{sec:strong} we state a stronger theorem (Theorem~\ref{thm:psd version}) of Batson-Spielman-Srivastava~\cite{BSS}, and prove that it implies Theorem~\ref{thm:BSS}. Section~\ref{sec:proof} contains the Batson-Spielman-Srivastava proof of this theorem, which is based on a highly original iterative argument. Section~\ref{sec: john} contains an application of Theorem~\ref{thm:psd version}, due to Srivastava~\cite{Sr1}, to approximate John decompositions. In section~\ref{sec:dim} we describe two applications of Theorem~\ref{thm:psd version}, due to Newman-Rabinovich~\cite{NR} and Schechtman~\cite{Sc3}, to dimensionality reduction problems. Section~\ref{sec:BT} describes the work of Spielman-Srivastava~\cite{SS2} that shows how their proof technique for Theorem~\ref{thm:psd version} can be used to prove a sharper version of the Bourgain-Tzafriri restricted invertibility principle. Section~\ref{sec:other} contains concluding comments and some open problems.

\section{A stronger theorem}\label{sec:strong}

Batson-Spielman-Srivastava actually proved a stronger theorem that implies Theorem~\ref{thm:BSS}. The statement below is not identical to the statement in~\cite{BSS}, though it easily follows from it. This formulation is stated explicitly as Theorem 1.6 in Srivastava's dissertation~\cite{Sr2}.

\begin{theo}\label{thm:psd version} Fix $\e\in (0,1)$ and $m,n\in \N$.
For every $x_1,\ldots,x_m\in \R^n$ there exist $s_1,\ldots,s_m\in [0,\infty)$ such that
\begin{equation}\label{eq:card s}
\left|\big\{i\in \{1,\ldots,m\}: s_i\neq 0\big\}\right|\le \left\lceil \frac{n}{\e^2}\right\rceil,
\end{equation}
and for all $y\in \R^n$ we have
\begin{equation}\label{eq:sparse tenson}
(1-\e)^2\sum_{i=1}^m \langle x_i,y\rangle^2\le \sum_{i=1}^m s_i\langle x_i,y\rangle^2\le (1+\e)^2 \sum_{i=1}^m \langle x_i,y\rangle^2.
\end{equation}
\end{theo}

\subsection{Deduction of Theorem~\ref{thm:BSS} from Theorem~\ref{thm:psd version}} Let $G=(g_{ij})$ be an $n\times n$ matrix with nonnegative entries. Note that the diagonal entries of $G$ play no role in the conclusion of Theorem~\ref{thm:BSS}, so we may assume in what follows that $g_{ii}=0$ for all $i\in \{1,\ldots,n\}$.


The degree matrix associated to $G$ is defined as usual by
\begin{equation}\label{eq:def diagonal matrix}D_G=
\begin{pmatrix} \sum_{j=1}^n g_{1j} & 0 & \dots& \dots&0 \\
  0 & \sum_{j=1}^n g_{2j}& \ddots& \ddots & \vdots\\
  \vdots & \ddots & \sum_{j=1}^n g_{3j}  & \ddots & \vdots\\
            \vdots & \ddots & \ddots& \ddots &0\\
              0 & \dots & \dots &0&\sum_{j=1}^n g_{nj}
                       \end{pmatrix},\end{equation}
and the Laplacian associated to $G$ is defined by
\begin{equation}\label{eq:def laplacian}
\Delta_G=D_G-G=\frac12\sum_{i=1}^n\sum_{j=1}^n g_{ij} (e_i-e_j)\otimes (e_i-e_j),
\end{equation}
where $e_1,\ldots,e_n\in \R^n$ is the standard basis of $\R^n$. In the last equation in~\eqref{eq:def laplacian}, and in what follows, we use standard tensor notation: for $x,y\in \R^n$ the linear operator $x\otimes y:\R^n\to \R^n$ is given by $(x\otimes y)(z)=\langle x,z\rangle y$.

Theorem~\ref{thm:psd version}, applied to the vectors $\{\sqrt{g_{ij}}\left(e_i-e_j\right):\ i,j\in \{1,\ldots,n\}\ \wedge \ i<j\}\subseteq \R^n$, implies that there exist $\{s_{ij}:\ i,j\in \{1,\ldots,n\}\ \wedge \ i<j\}\subseteq [0,\infty)$, at most $\left\lceil n/\e^2\right\rceil$ of which are nonzero, such that for every $y\in \R^n$ we have
\begin{equation}\label{eq:sparsified}
\left\langle\Delta_Gy,y\right\rangle\le \sum_{i=1}^{n-1}\sum_{j=i+1}^n s_{ij}g_{ij}\left\langle e_i-e_j,y\right\rangle^2\le \left(\frac{1+\e}{1-\e}\right)^2 \left\langle\Delta_Gy,y\right\rangle.
\end{equation}
Extend $(s_{ij})_{i<j}$ to a symmetric matrix by setting $s_{ii}=0$ and $s_{ji}=s_{ij}$ if $i>j$, and define $H=(h_{ij})$ by $h_{ij}=s_{ij}g_{ij}$. Then $\supp(H)\subseteq \supp(G)$ and $|\supp(H)|\le 2 \left\lceil n/\e^2\right\rceil$. A straightforward computation shows that $\left\langle\Delta_Gy,y\right\rangle=\frac12\sum_{i=1}^n\sum_{j=1}^ng_{ij}(y_i-y_j)^2$ and $\sum_{i=1}^{n-1}\sum_{j=i+1}^n s_{ij}g_{ij}\left\langle e_i-e_j,y\right\rangle^2=\frac12\sum_{i=1}^n\sum_{j=1}^nh_{ij}(y_i-y_j)^2$. Thus, due to~\eqref{eq:sparsified} Theorem~\ref{thm:BSS} follows, with the bound on $|\supp(H)|$ as in~\eqref{eq:c(eps)}.\qed




\section{Proof of Theorem 2.1}\label{sec:proof}

Write $A=\sum_{i=1}^m x_i\otimes x_i$. Note that  it suffices to prove Theorem~\ref{thm:psd version} when $A$ is the $n\times n$ identity matrix $I$. Indeed, by applying an arbitrarily small perturbation we may assume that $A$ is invertible. If we then set $y_i=A^{-1/2}x_i$ then $\sum_{i=1}^m y_i\otimes y_i=I$, and the conclusion of Theorem~\ref{thm:psd version} for the vectors $\{y_1,\ldots,y_m\}$ implies the corresponding conclusion for the original vectors $\{x_1,\ldots,x_m\}$.

The situation is therefore as follows. We are given $x_1,\ldots,x_n\in \R^n$ satisfying
\begin{equation}\label{eq:decompositon of identity}
\sum_{i=1}^m x_i\otimes x_i=I.
\end{equation}
Our goal is to find $\{s_i\}_{i=1}^m\subseteq [0,\infty)$ such that at most $\lceil n/\e^2\rceil$ of them are nonzero, and
\begin{equation}\label{eq:spectral goal}
 \frac{\lambda_1\left(\sum_{i=1}^n s_ix_i\otimes x_i\right)}{\lambda_n\left(\sum_{i=1}^n s_ix_i\otimes x_i\right)}\le \left(\frac{1+\e}{1-\e}\right)^2.
\end{equation}

For the ensuing argument it will be convenient to introduce the following notation:
\begin{equation}\label{eq:def theta}
\theta=\frac{1+\e}{1-\e}.
\end{equation}

The proof constructs by induction $\{t_k\}_{k=1}^\infty\subseteq [0,\infty)$ and $\{y_k\}_{k=1}^\infty\subseteq \{x_1,\ldots,x_m\}$ with the following properties. Setting $A_0=0$ and $A_i=\sum_{j=1}^i t_j y_j\otimes y_j$ for $i\in \N$, the following inequalities hold true:
\begin{equation}\label{eq:spectal inductive}
-\frac{n}{\e}+i< \lambda_n(A_i)\le \lambda_1(A_i)<\theta\left(\frac{n}{\e}+i\right),
\end{equation}
and for every $i\in \N$ we have
\begin{equation}\label{eq:upper inductive}
\sum_{j=1}^n \frac{1}{\theta\left(\frac{n}{\e}+i\right)-\lambda_j(A_i)}= \sum_{j=1}^n \frac{1}{\theta\left(\frac{n}{\e}+i-1\right)-\lambda_j(A_{i-1})},
\end{equation}
and
\begin{equation}\label{eq:lower inductive}
\sum_{j=1}^n \frac{1}{\lambda_j(A_i)-\left(-\frac{n}{\e}+i\right)}\le \sum_{j=1}^n \frac{1}{\lambda_j(A_{i-1})-\left(-\frac{n}{\e}+i-1\right)}.
\end{equation}
(The sums in~\eqref{eq:upper inductive} and~\eqref{eq:lower inductive} represent the traces of certain matrices constructed from the $A_i$, and we will soon see that this is the source of their relevance.)

If we continue this construction for $k=\lceil n/\e^2\rceil$ steps, then by virtue of~\eqref{eq:spectal inductive} we would have
$$
\frac{\lambda_1(A_k)}{\lambda_n(A_k)}\le \frac{\theta\left(\frac{n}{\e}+\frac{n}{\e^2}\right)}{\frac{n}{\e^2}-\frac{n}{\e}}=\left(\frac{1+\e}{1-\e}\right)^2.
$$
By construction $A_{k}=\sum_{i=1}^m s_ix_i\otimes x_i$ with $s_1,\ldots,s_m\in [0,\infty)$ and at most $k$ of them nonzero. Thus, this process would prove the desired inequality~\eqref{eq:spectral goal}.

Note that while for our purposes we just need the spectral bounds in~\eqref{eq:spectal inductive}, we will need the additional conditions on the resolvent appearing in~\eqref{eq:upper inductive} and~\eqref{eq:lower inductive} in order for us to be able to perform the induction step. Note also that due to~\eqref{eq:spectal inductive} all the summands in~\eqref{eq:upper inductive} and~\eqref{eq:lower inductive} are positive.

Suppose that $i\ge 1$ and we have already constructed the scalars $t_1,\ldots,t_{i-1}\in [0,\infty)$ and vectors $y_1,\ldots,y_{i-1}\in \{x_1,\ldots,x_m\}$, and let $A_{i-1}$ be the corresponding positive semidefinite matrix. The proof of Theorem~\ref{thm:psd version} will be complete once we show that we can find $t_i\ge 0$ and $y_i\in \{x_1,\ldots,x_m\}$ so that the matrix $A_i=A_{i-1}+t_iy_i\otimes y_i$ satisfies the conditions \eqref{eq:spectal inductive}, \eqref{eq:upper inductive}, \eqref{eq:lower inductive}.

It follows from the inductive hypotheses~\eqref{eq:spectal inductive} and~\eqref{eq:lower inductive} that
\begin{multline}\label{eq:less than eps}
0<\frac{1}{\lambda_n(A_{i-1})-\left(-\frac{n}{\e}+i-1\right)}\le\sum_{j=1}^n \frac{1}{\lambda_j(A_{i-1})-\left(-\frac{n}{\e}+i-1\right)}\\\le \sum_{j=1}^n \frac{1}{\lambda_j(A_{0})-\left(-\frac{n}{\e}\right)}=\e<1.
\end{multline}
Hence, since $A_i-A_{i-1}$ is positive semidefinite,
$
\lambda_n(A_i)\ge \lambda_n(A_{i-1})> -\frac{n}{\e}+i,
$
implying the leftmost inequality in~\eqref{eq:spectal inductive}.

It will be convenient to introduce the following notation:
\begin{equation}\label{eq:def a}
a= \sum_{j=1}^n \frac{1}{\theta\left(\frac{n}{\e}+i-1\right)-\lambda_j(A_{i-1})}-\sum_{j=1}^n \frac{1}{\theta\left(\frac{n}{\e}+i\right)-\lambda_j(A_{i-1})}>0,
\end{equation}
and
\begin{equation}\label{eq:def b}
b= \sum_{j=1}^n \frac{1}{\lambda_j(A_{i-1})-\left(-\frac{n}{\e}+i\right)}-\sum_{j=1}^n \frac{1}{\lambda_j(A_{i-1})-\left(-\frac{n}{\e}+i-1\right)} >0.
\end{equation}
Note that~\eqref{eq:def b} makes sense since, as we have just seen, \eqref{eq:less than eps} implies  that we have $\lambda_n(A_{i-1})> -\frac{n}{\e}+i$. This, combined with~\eqref{eq:spectal inductive}, shows that the matrices $\theta\left(\frac{n}{\e}+i\right)I-A_{i-1}$ and $A_{i-1}-\left(-\frac{n}{\e}+i\right)I$ are positive definite, and hence also invertible. Therefore, for every $j\in \{1,\ldots,m\}$ we can consider the following quantities:
\begin{equation}\label{eq:alpha_j}
\alpha_j=\left\langle\left(\theta\left(\frac{n}{\e}+i\right)I-A_{i-1}\right)^{-1}x_j,x_j\right\rangle+\frac{1}{a}
\left\langle\left(\theta\left(\frac{n}{\e}+i\right)I-A_{i-1}\right)^{-2}x_j,x_j\right\rangle,
\end{equation}
and
\begin{equation}\label{eq:beta_j}
\beta_j=\frac{1}{b}
\left\langle\left(A_{i-1}-\left(-\frac{n}{\e}+i\right)I\right)^{-2}x_j,x_j\right\rangle-
\left\langle\left(A_{i-1}-\left(-\frac{n}{\e}+i\right)I\right)^{-1}x_j,x_j\right\rangle.
\end{equation}

The following lemma contains a crucial inequality between these quantities.
\begin{lemm}\label{lem:averaged}
We have $\sum_{j=1}^m\beta_j\ge \sum_{j=1}^m\alpha_j$.
\end{lemm}

Assuming Lemma~\ref{lem:averaged} for the moment, we will show now how to complete the inductive construction. By Lemma~\ref{lem:averaged} there exists $j\in \{1,\ldots,m\}$ for which $\beta_j\ge \alpha_j$. We will fix this $j$ from now on. Denote
\begin{equation}\label{eq:the inductive def}
t_i=\frac{1}{\alpha_j}\quad \mathrm{and}\quad y_i=x_j.
\end{equation}

The following formula is straightforward to verify---it is known as the Sherman-Morrison formula (see~\cite[Section~2.1.3]{GV}): for every invertible $n\times n$ matrix $A$ and every $z\in \R^n$ we have
\begin{equation}\label{eq:SM}
\left(A+z\otimes z\right)^{-1}= A^{-1}-\frac{1}{1+\left\langle A^{-1}z,z\right\rangle}A^{-1}(z\otimes z)A^{-1}.
\end{equation}
Note that $\trace\left(A^{-1}(z\otimes z)A^{-1}\right)=\left\langle A^{-2}z,z\right\rangle$. Hence, by taking the trace of the identity~\eqref{eq:SM} we have
\begin{equation}\label{eq:trace}
\trace\left(\left(A+z\otimes z\right)^{-1}\right)=\trace\left(A^{-1}\right)-\frac{\left\langle A^{-2}z,z\right\rangle}{1+\left\langle A^{-1}z,z\right\rangle}.
\end{equation}

Now, for every $t\in (0,1/\alpha_j]$ we have
\begin{eqnarray}
&&\!\!\!\!\!\!\!\!\!\!\!\!\!\!\!\nonumber\sum_{j=1}^n \frac{1}{\theta\left(\frac{n}{\e}+i\right)-\lambda_j(A_{i-1}+t x_j\otimes x_j)}=\trace\left(\left(\theta\left(\frac{n}{\e}+i\right)I-A_{i-1}-tx_j\otimes x_j\right)^{-1}\right)\\\nonumber
&\stackrel{\eqref{eq:trace}}{=}& \sum_{j=1}^n \frac{1}{\theta\left(\frac{n}{\e}+i\right)-\lambda_j(A_{i-1})}+
\frac{\left\langle\left(\theta\left(\frac{n}{\e}+i\right)I-A_{i-1}\right)^{-2}
x_j,x_j\right\rangle}
{\frac{1}{t}-\left\langle\left(\theta\left(\frac{n}{\e}+i\right)I-A_{i-1}\right)^{-1}x_j,x_j\right\rangle}\\&\le& \label{eq:with equality} \sum_{j=1}^n \frac{1}{\theta\left(\frac{n}{\e}+i\right)-\lambda_j(A_{i-1})}
+\frac{\left\langle\left(\theta\left(\frac{n}{\e}+i\right)I-A_{i-1}\right)^{-2}
x_j,x_j\right\rangle}
{\alpha_j-\left\langle\left(\theta\left(\frac{n}{\e}+i\right)I-A_{i-1}\right)^{-1}x_j,x_j\right\rangle}\\\nonumber
&\stackrel{\eqref{eq:alpha_j}}{=}& \sum_{j=1}^n \frac{1}{\theta\left(\frac{n}{\e}+i\right)-\lambda_j(A_{i-1})}+a
\\&\stackrel{\eqref{eq:def a}}{=}& \sum_{j=1}^n \frac{1}{\theta\left(\frac{n}{\e}+i-1\right)-\lambda_j(A_{i-1})}\label{eq:finite}.
\end{eqnarray}
In~\eqref{eq:with equality} we used the fact that $t\le 1/\alpha_j$ and  $\alpha_j> \left\langle\left(\theta\left(\frac{n}{\e}+i\right)I-A_{i-1}\right)^{-1}x_j,x_j\right\rangle$. In particular, there is equality in~\eqref{eq:with equality} if $t=1/\alpha_j$. As $A_i=A_{i-1}+\frac{1}{\alpha_j} x_j\otimes x_j$, this proves~\eqref{eq:upper inductive}. Inequality~\eqref{eq:finite} also implies the rightmost inequality in~\eqref{eq:spectal inductive}. Indeed, assume for contradiction that $\lambda_1\left(A_{i-1}+\frac{1}{\alpha_j}x_j\otimes x_j\right)\ge \theta\left(\frac{n}{\e}+i\right)$. Since by the inductive hypothesis $\lambda_1(A_{i-1})< \theta\left(\frac{n}{\e}+i-1\right)<\theta\left(\frac{n}{\e}+i\right)$, it follows by continuity that there exists $t\in (0, 1/\alpha_j]$ for which $\lambda_1\left(A_{i-1}+tx_j\otimes x_j\right)= \theta\left(\frac{n}{\e}+i\right)$. This value of $t$ would make $\sum_{j=1}^n1/\left(\theta\left(\frac{n}{\e}+i\right)-\lambda_j(A_{i-1}+t x_j\otimes x_j)\right)$ be infinite, contradicting~\eqref{eq:finite} since by the inductive hypothesis all the summands in the right-hand side of~\eqref{eq:finite} are positive and finite.

It remains to prove~\eqref{eq:lower inductive}---this is the only place where the condition $\beta_j\ge \alpha_j$ will be used. We proceed as follows.
\begin{eqnarray*}
&&\!\!\!\!\!\!\!\!\!\!\!\!\!\!\!\sum_{j=1}^n \frac{1}{\lambda_j(A_i)-\left(-\frac{n}{\e}+i\right)}\stackrel{\eqref{eq:the inductive def}}{=}\trace\left(\left(A_{i-1}-\left(-\frac{n}{\e}+i\right)I+\frac{1}{\alpha_j}x_j\otimes x_j\right)^{-1}\right)\\
&\stackrel{\eqref{eq:trace}}{=}& \sum_{j=1}^n \frac{1}{\lambda_j(A_{i-1})-\left(-\frac{n}{\e}+i\right)}-\frac{\left\langle\left(A_{i-1}-
\left(-\frac{n}{\e}+i\right)I\right)^{-2}x_j,x_j\right\rangle}{\alpha_j+\left\langle\left(A_{i-1}-\left(-\frac{n}{\e}+i\right)I\right)^{-1}x_j,x_j\right\rangle}\\
&\stackrel{(\beta_j\ge \alpha_j)}{\le}& \sum_{j=1}^n \frac{1}{\lambda_j(A_{i-1})-\left(-\frac{n}{\e}+i\right)}-\frac{\left\langle\left(A_{i-1}-
\left(-\frac{n}{\e}+i\right)I\right)^{-2}x_j,x_j\right\rangle}{\beta_j+\left\langle\left(A_{i-1}-\left(-\frac{n}{\e}+i\right)I\right)^{-1}x_j,x_j\right\rangle}\\
&\stackrel{\eqref{eq:beta_j}}{=}& \sum_{j=1}^n\frac{1}{\lambda_j(A_{i-1})-\left(-\frac{n}{\e}+i\right)}-b\\
&\stackrel{\eqref{eq:def b}}{=}& \sum_{j=1}^n \frac{1}{\lambda_j(A_{i-1})-\left(-\frac{n}{\e}+i-1\right)}.
\end{eqnarray*}
This concludes the inductive construction, and hence also the proof of Theorem~\ref{thm:psd version}, provided of course that we prove the crucial inequality contained in Lemma~\ref{lem:averaged}.
\begin{proof}[Proof of Lemma~\ref{lem:averaged}] It is straightforward to check that the identity~\eqref{eq:decompositon of identity} implies that for every $n\times n$ matrix $A$ we have
\begin{equation}\label{eq:trace identity}
\sum_{j=1}^m \left\langle Ax_j,x_j\right \rangle=\trace(A).
\end{equation}
Hence,
\begin{equation}\label{eq:sum alpha}
\sum_{j=1}^m\alpha_j\stackrel{\eqref{eq:alpha_j}\wedge\eqref{eq:trace identity}}{=}\tr\left(\left(\theta\left(\frac{n}{\e}+i\right)I-A_{i-1}\right)^{-1}\right)+
\frac{\tr\left(\left(\theta\left(\frac{n}{\e}+i\right)I-A_{i-1}\right)^{-2}\right)}{a},
\end{equation}
and,
\begin{equation}\label{eq:sum beta}
\sum_{j=1}^m\beta_j\stackrel{\eqref{eq:beta_j}\wedge\eqref{eq:trace identity}}{=}\frac{\trace\left(\left(A_{i-1}-\left(-\frac{n}{\e}+i\right)I\right)^{-2}\right)}{b}-
 \trace\left(\left(A_{i-1}-\left(-\frac{n}{\e}+i\right)I\right)^{-1}\right).
\end{equation}

Now,
\begin{multline}\label{eq:first trace alpha}
\trace\left(\left(\theta\left(\frac{n}{\e}+i\right)I-A_{i-1}\right)^{-1}\right)=\sum_{j=1}^n \frac{1}{\theta\left(\frac{n}{\e}+i\right)-\lambda_j(A_{i-1})}
\\\le \sum_{j=1}^n \frac{1}{\theta\left(\frac{n}{\e}+i-1\right)-\lambda_j(A_{i-1})}\stackrel{\eqref{eq:upper inductive}}{=}\sum_{j=1}^n \frac{1}{\frac{\theta n}{\e}-\lambda_j(A_{0})}=\frac{\e}{\theta},
\end{multline}
and
\begin{multline}\label{eq:second trace alpha}
\frac{1}{a}\cdot \trace\left(\left(\theta\left(\frac{n}{\e}+i\right)I-A_{i-1}\right)^{-2}\right)\\\stackrel{\eqref{eq:def a}}{=}\frac{\sum_{j=1}^n \left(\theta\left(\frac{n}{\e}+i\right)-\lambda_j(A_{i-1})\right)^{-2}}{\theta\sum_{j=1}^n\left(\theta\left(\frac{n}{\e}+i\right)-
\lambda_j(A_{i-1})\right)^{-1}\left(\theta\left(\frac{n}{\e}+i-1\right)-\lambda_j(A_{i-1})\right)^{-1}}\le \frac{1}{\theta}.
\end{multline}
Hence,
\begin{equation}\label{eq:upper sum alpha}
\sum_{j=1}^n \alpha_j\stackrel{\eqref{eq:sum alpha}\wedge\eqref{eq:first trace alpha}\wedge\eqref{eq:second trace alpha}}{\le} \frac{1+\e}{\theta}\stackrel{\eqref{eq:def theta}}{=} 1-\e.
\end{equation}

In order to use~\eqref{eq:sum beta}, we first bound $b$ as follows.
\begin{eqnarray*}
\nonumber b&\stackrel{\eqref{eq:def b}}{=}&\sum_{j=1}^n\frac{1}{\left(\lambda_{j-1}(A_{i-1})-\left(-\frac{n}{\e}+i\right)\right)
\left(\lambda_{j-1}(A_{i-1})-\left(-\frac{n}{\e}+i-1\right)\right)}\\\nonumber
&\le& \left(\sum_{j=1}^n \frac{1}{\lambda_j(A_{i-1})-\left(-\frac{n}{\e}+i-1\right)}\right)^{1/2}\\\nonumber&&\quad\cdot\left(\sum_{j=1}^n\frac{1}{\left(\lambda_{j-1}(A_{i-1})-\left(-\frac{n}{\e}+i\right)\right)^2
\left(\lambda_{j-1}(A_{i-1})-\left(-\frac{n}{\e}+i-1\right)\right)}\right)^{1/2}\\\nonumber
&\stackrel{\eqref{eq:less than eps}}{\le}&\sqrt{\e}\left(\sum_{j=1}^n\frac{1}{\left(\lambda_{j-1}(A_{i-1})-\left(-\frac{n}{\e}+i\right)\right)^{2}}-b\right)^{1/2}\\
&\le &\left(\sum_{j=1}^n\frac{1}{\left(\lambda_{j-1}(A_{i-1})-\left(-\frac{n}{\e}+i\right)\right)^{2}}-b\right)^{1/2},
\end{eqnarray*}
which simplifies to give the bound
\begin{equation}\label{eq:use CW}
\frac{1}{b}\sum_{j=1}^n\frac{1}{\left(\lambda_{j-1}(A_{i-1})-\left(-\frac{n}{\e}+i\right)\right)^{2}}=
\frac{\trace\left(\left(A_{i-1}-\left(-\frac{n}{\e}+i\right)I\right)^{-2}\right)}{b}\ge b+1.
\end{equation}
Hence,
\begin{multline}\label{lower sum beta}
\sum_{j=1}^m \beta_j \stackrel{\eqref{eq:sum beta}\wedge\eqref{eq:use CW}}{\ge} b+1-\sum_{j=1}^n \frac{1}{\lambda_j(A_{i-1})-\left(-\frac{n}{\e}+i\right)}\\
\stackrel{\eqref{eq:def b}}{=} 1-\sum_{j=1}^n \frac{1}{\lambda_j(A_{i-1})-\left(-\frac{n}{\e}+i-1\right)}\stackrel{\eqref{eq:less than eps}}\ge 1-\e.
\end{multline}
Lemma~\ref{lem:averaged} now follows from~\eqref{eq:upper sum alpha} and~\eqref{lower sum beta}.
\end{proof}

\begin{rema}
In the inductive construction, instead of ensuring equality in~\eqref{eq:upper inductive}, we could have ensured equality in~\eqref{eq:lower inductive} and replaced the equality sign in~\eqref{eq:upper inductive} with the inequality sign $\le$. This would be achieved by choosing $t_i=1/\beta_j$ in~\eqref{eq:the inductive def}. Alternatively we could have chosen $t_i$ to be any value in the interval $[1/\beta_j,1/\alpha_j]$, in which case both inductive conditions~\eqref{eq:upper inductive} and \eqref{eq:lower inductive} would be with the inequality sign $\le$.
\end{rema}

\section{Approximate John decompositions}\label{sec: john}

\newcommand{\BM}{\mathrm{BM}}

Let $B_2^n\subseteq \R^n$ be the  unit ball with respect to the standard Euclidean metric. Recall that an ellipsoid $\mathcal E=TB_2^n\subseteq \R^n$ is an image of $B_2^n$ under an invertible linear transformation $T:\R^n\to \R^n$. Let $K\subseteq \R^n$ be a centrally symmetric (i.e., $K=-K$) convex body. John's theorem~\cite{Jo} states that among the ellipsoids that contain $K$, there exists a unique ellipsoid of minimal volume. This ellipsoid is called the John ellipsoid of $K$. If the John ellipsoid of $K$ happens to be $B_2^n$, the body $K$ is said to be in John position. For any $K$ there is a linear invertible transformation $T:\R^n\to \R^n$ such that $TK$ is in John position. The Banach-Mazur distance between two centrally symmetric convex bodies $K,L\subseteq \R^n$, denoted $d_{\BM}(K,L)$, is the infimum over those $s>0$ for which there exists a linear operator $T:\R^n\to \R^n$ satisfying $K\subseteq TL\subseteq sK$.

 John~\cite{Jo} proved that if $K$ is in John position then there exist contact points $x_1,\ldots,x_m\in (\partial K)\cap (\partial B_2^n)$ and positive weights $c_1,\ldots,c_m>0$ such that
\begin{equation}\label{eq:center of mass}
\sum_{i=1}^m c_i x_i=0,
\end{equation}
and
\begin{equation}\label{eq:john decomposition}
\sum_{i=1}^m c_i x_i\otimes x_i=I.
\end{equation}
When conditions~\eqref{eq:center of mass} and~\eqref{eq:john decomposition} are satisfied we say that $\{x_i,c_i\}_{i=1}^m$ form a John decomposition of the identity. It is hard to overstate the importance of John decompositions in analysis and geometry, and we will not attempt to discuss their applications here. Interested readers are referred to~\cite{Bal2} for a taste of this rich field.

John proved that one can always take $m\le n(n+1)/2$. This bound cannot be improved in general (see~\cite{PT} for an even stronger result of this type). However, if one allows an arbitrarily small perturbation of the body $K$, it is possible to reduce the number of contact points with the John ellipsoid to grow linearly in $n$. This sharp result is a consequence of the Batson-Spielman-Srivastava sparsification theorem~\ref{thm:psd version}, and it was proved by Srivastava in~\cite{Sr1}. The precise formulation of Srivastava's theorem is as follows.

\begin{theo}\label{thm:sr john}
 If $K\subseteq \R^n$ is a centrally symmetric convex body and $\e\in (0,1)$ then there exists a convex body $L\subseteq \R^n$ with $d_{\BM}(K,L)\le 1+\e$ such that $L$ has at most $m= O(n/\e^2)$ contact points with its John ellipsoid.
\end{theo}
The problem of perturbing a convex body so as to reduce the size of its John decomposition was studied by Rudelson in~\cite{Ru1}, where the bound $m\le C(\e)n(\log n)^3$ was obtained via a randomized construction. In~\cite{Ru2} Rudelson announced an improved bound of $m\le C(\e)n\log n(\log\log n)^2$ using a different probabilistic argument based on majorizing measures, and in~\cite{Ru3} Rudelson obtained the bound $m=O(\e^{-2}n\log n)$, which was the best known bound prior to Srivastava's work.

The key step in all of these proofs is to extract from~\eqref{eq:john decomposition} an approximate John decomposition. This amounts to finding weights $s_1,\ldots,s_m\in [0,\infty)$, such that not many of them are nonzero, and such that we have the operator norm bound $\left\|I-\sum_{i=1}^ms_i x_i\otimes x_i\right\|\le \e.$ This is exactly what Theorem~\ref{thm:psd version} achieves, with $|\{i\in \{1,\ldots, m\}:\ s_i\neq 0\}|\le c(\e)n$. Prior to the deterministic construction of Batson-Spielman-Strivastava~\cite{BSS}, such approximate John decompositions were constructed by Rudelson via a random selection argument, and a corresponding operator-valued  concentration inequality. In particular, Rudelson's bound~\cite{Ru3} $m=O(\e^{-2}n\log n)$ uses an influential argument of Pisier. Such methods are important to a variety of applications (see~\cite{RV,Tr2}), and in particular this is how Spielman-Srivastava~\cite{SS1} proved their earlier $O(\e^{-2}n\log n)$ sparsification theorem. While yielding suboptimal results, this method is important since it has almost linear (randomized) running time. We refer to the recent work of  Adamczak, Litvak, Pajor and Tomczak-Jaegermann for deeper investigations of randomized approximations of certain decompositions of the identity (under additional assumptions).

\begin{proof}[Proof of Theorem~\ref{thm:sr john}]
Suppose that $K$ is in John position, and let $\{x_i,c_i\}_{i=1}^n$ be the corresponding John-decomposition. Since $\sum_{i=1}^m (\sqrt{c_i} x_i)\otimes (\sqrt{c_i} x_i)=I$, we may use Theorem~\ref{thm:psd version} to find $s_1,\ldots, s_m\ge 0$, with at most $O(n/\e^2)$ of them nonzero, such that if we set $A=\sum_{i=1}^m s_i c_i x_i\otimes x_i$, then the matrices $A-I$ and $(1+\e/4)I-A$ are positive semidefinite. Thus $\|A-I\|\le \e/4$.


The rest of the proof follows the argument in~\cite{Ru1,Ru2}. Write $\mathcal E=A^{1/2}B_2^n$. Then since $\|A-I||\le \e/4$ we have
$$
\left(1-\frac{\e}{4}\right)\mathcal{E}\subseteq B_2^n\subseteq \left(1+\frac{\e}{4}\right)\mathcal{E}.
$$
Denote $y_i=x_i/\|A^{-1/2}x_i\|_2\in \partial\mathcal{E}$ and define
$$
H=\mathrm{conv}\left(\{\pm y_i\}_{i\in J}\bigcup\left(\frac{1}{1+\e}K\right)\right),
$$
where $J=\{i\in \{1,\ldots, m\}: s_i\neq 0\}$. Then $H$ is a centrally symmetric convex body, and by a straightforward argument one checks (see~\cite{Ru1,Ru2}) that $\frac{1}{1+\e}K\subseteq H\subseteq (1+2\e)K$.

Set $L=A^{-1/2}H$. Since $K\subseteq B_2^n$ we have $(\partial H)\cap (\partial\mathcal{E})= \{\pm y_i\}_{i\in J}$, and therefore $(\partial L)\cap (\partial B_2^n)= \{\pm z_i\}_{i\in J}$, where $z_i=A^{-1/2}y_i$. Writing $a_i=\frac{c_is_i}{2}\|A^{1/2}x_i\|_2$, we have
\begin{multline*}
\sum_{i\in J} a_iz_i\otimes z_i+\sum_{i\in J} a_i(-z_i)\otimes (-z_i)=\sum_{i=1}^m s_ic_i (A^{-1/2}x_i)\otimes (A^{-1/2}x_i)\\
= A^{-1/2}\left(\sum_{i=1}^m s_ic_i x_i\otimes x_i\right)A^{-1/2}=A^{-1/2}AA^{-1/2}=I.
\end{multline*}
Hence $\{\pm z_i, a_i\}_{i\in J}$ form a John decomposition of the identity consisting of contact points of $L$ and $B_2^n\supseteq L$. By John's uniqueness theorem~\cite{Jo} it follows that $B_2^n$ is the John ellipsoid of $L$.
\end{proof}

\begin{rema}
Rudelson~\cite{Ru2,Ru3} also studied approximate John decompositions for non-centrally symmetric convex bodies. He proved that Theorem~\ref{thm:sr john} holds if $K$ is not necessarily centrally symmetric, with $m=O(\e^{-2}n\log n)$. Note that in the non-symmetric setting one needs to define the Banach-Mazur appropriately: $d_{\BM}(K,L)$ is the infimum over those $s>0$ for which there exists $v\in \R^n$ and a linear operator $T:\R^n\to \R^n$ satisfying $K+v\subseteq TL\subseteq s(K+v)$. Srivastava~\cite{Sr1}, based on a refinement of the proof technique of Theorem~\ref{thm:psd version}, proved that  if $K\subseteq \R^n$ is a convex body and $\e\in (0,1)$,  then there exists a convex body $L\subseteq \R^n$ with $d_{\BM}(K,L)\le \sqrt{5}+\e$ such that $L$ has at most $m= O(n/\e^3)$ contact points with its John ellipsoid. Thus, it is possible to get bounded perturbations with linearly many contact points with the John ellipsoid, but it remains open whether this is possible with $1+\e$ perturbations. The problem is how to ensure condition~\eqref{eq:center of mass} for an approximate John decomposition using the Batson-Spielman-Srivastava technique---for symmetric bodies this is not a problem since we can take the reflections of the points in the approximate John decomposition.
\end{rema}

\section{Dimensionality reduction in $L_p$ spaces}\label{sec:dim}

Fix $p\ge 1$. In what follows $L_p$ denotes the space of $p$-integrable functions on $[0,1]$ (equipped with Lebesgue measure), and $\ell_p^n$ denotes the space $\R^n$, equipped with the $\ell_p$ norm $\|x\|_p=\left(\sum_{i=1}^n |x_i|^p\right)^{1/p}$. Since any $n$-dimensional subspace of $L_2$ is isometric to $\ell_2^n$, for any $x_1,\ldots,x_n\in L_2$ there exist $y_1,\ldots,y_n\in \ell_2^n$ satisfying $\|x_i-x_j\|_2=\|y_i-y_j\|_2$ for all $i,j\in \{1,\ldots,n\}$. But, more is true if we allow errors: the Johnson-Lindenstrauss lemma~\cite{JL} says that for every $x_1,\ldots,x_n\in L_2$, $\e\in (0,1)$ there exists $k=O(\e^{-2}\log n)$ and $y_1,\ldots,y_n\in \ell_2^k$ such that $\|x_i-x_j\|_2\le \|y_i-y_j\|_2\le (1+\e)\|x_i-x_j\|_2$ for all $i,j\in \{1,\ldots,n\}$. This bound on $k$ is known to be sharp up to a $O(\log (1/\e))$ factor~\cite{Al}.

In $L_p$ for $p\neq 2$ the situation is much more mysterious. Any $n$-points in $L_p$ embed isometrically into $\ell_p^k$ for $k=n(n-1)/2$, and this bound on $k$ is almost optimal~\cite{Bal1}. If one is interested, as in the Johnson-Lindenstrauss lemma, in embeddings of $n$-point subsets of $L_p$ into $\ell_p^k$ with a $1+\e$ multiplicative error in the pairwise distances, then the best known bound on $k$, due to Schechtman~\cite{Sc2}, was
\begin{equation}\label{eq:cases}
k\le \left\{
\begin{array}{ll}
C(\e)n\log n & p\in [1,2),\\
C(p,\e)n^{p/2}\log n & p\in (2,\infty).
\end{array}
\right.
\end{equation}

We will see now how Theorem~\ref{thm:psd version} implies improvements to the bounds in~\eqref{eq:cases} when $p=1$ and when $p$ is an even integer. The bounds in~\eqref{eq:cases} for $p\notin \{1\}\cup 2\N$ remain the best currently known. We will start with the improvement when $p=1$, which is due to Newman and Rabinovich~\cite{NR}. In the case $p\in 2\N$, which is due to Schechtman~\cite{Sc3}, more is true: the claimed bound on $k$ holds for embeddings of any $n$-dimensional linear subspace of $L_p$ into $\ell_p^k$, and when stated this way (rather than for $n$-point subsets of $L_p$) it is sharp~\cite{BDGJN}.

\subsection{Finite subsets of $L_1$}\label{sec:L1}

It is known that a Johnson-Lindenstrauss type result cannot hold in $L_1$: Brinkman and Charikar~\cite{BC} proved that for any $D>1$ there exists arbitrarily large $n$-point subsets $\{x_1,\ldots,x_n\}\subseteq L_1$ with the property that if they embed with distortion $D$ into $\ell_1^k$ then necessarily $k\ge n^{c/D^2}$, where $c>0$ is a universal constant. Here, and in what follows, a metric space $(X,d)$ is said to embed with distortion $D$ into a normed space $Y$ if there exists $f:X\to Y$ satisfying $d(x,y)\le \|f(x)-f(y)\|\le Dd(x,y)$ for all $x,y\in X$. No nontrivial restrictions on bi-Lipschitz dimensionality reduction are known for finite subsets of $L_p$, $p\in (1,\infty)\setminus \{2\}$. On the positive side, as stated in~\eqref{eq:cases}, Schechtman proved~\cite{Sc2} that any $n$-point subset of $L_1$ embeds with distortion $1+\e$ into $\ell_1^k$, for some $k\le C(\e)n\log n$. The following theorem of Newman and Rabinovich~\cite{NR} gets the first asymptotic improvement over Schechtman's 1987 bound, and is based on the Batson-Spielman-Srivastava theorem.

\begin{theo}\label{thm:NR}
For any $\e\in (0,1)$, any $n$-point subset of $L_1$ embeds with distortion $1+\e$ into $\ell_1^k$ for some $k=O(n/\e^2)$.
\end{theo}

\begin{proof} Let $f_1,\ldots,f_n\in L_1$ be distinct. By the cut-cone representation of $L_1$ metrics, there exists nonnegative weights $\{w_E\}_{E\subseteq \{1,\ldots,n\}}$ such that for all $i,j\in\{1,\ldots,n\}$ we have
\begin{equation}\label{eq:cut cone}
\|f_i-f_j\|_1=\sum_{E\subseteq \{1,\ldots,n\}} w_E |\1_E(i)-\1_E(j)|.
\end{equation}
See~\cite{DL} for a proof of~\eqref{eq:cut cone} (see also~\cite[Section~3]{Na} for a quick proof).

For every $E\subseteq \{1,\ldots,n\}$ define $x_E=\sqrt{w_E}\sum_{i\in E} e_i\in \R^n$ ($e_1,\ldots,e_n$ is the standard basis of $\R^n$). By Theorem~\ref{thm:psd version} there exists a subset $\sigma\subseteq 2^{\{1,\ldots,n\}}$ with $|\sigma|=O(n/\e^2)$, and nonnegative weights $\{s_E\}_{E\in \sigma}$, such that for every $y\in \R^n$ we have
\begin{equation}\label{eq:cuts sparsified}
\sum_{E\subseteq \{1,\ldots,n\}} w_E \left(\sum_{i\in E} y_i\right)^2\le \sum_{E\in \sigma} s_Ew_E \left(\sum_{i\in E} y_i\right)^2\le (1+\e)\sum_{E\subseteq \{1,\ldots,n\}} w_E \left(\sum_{i\in E} y_i\right)^2.
\end{equation}

Define $z_1,\ldots,z_n\in \R^\sigma$ by $z_i=\left(s_Ew_E\1_E(i)\right)_{E\in \sigma}$. For $i,j\in \{1,\ldots, n\}$ apply~\eqref{eq:cuts sparsified} to the vector $y=e_i-e_j$, noting that for all $E\subseteq \{1,\ldots, n\}$, for this vector $y$ we have $\left(\sum_{i\in E} y_i\right)^2=|\1_E(i)-\1_E(j)|$.

\begin{multline*}
\|f_i-f_j\|_1\stackrel{\eqref{eq:cut cone}}{=}\sum_{E\subseteq \{1,\ldots,n\}} w_E |\1_E(i)-\1_E(j)|\stackrel{\eqref{eq:cuts sparsified}}{\le}
\sum_{E\in \sigma} s_Ew_E|\1_E(i)-\1_E(j)|\\=\|z_i-z_j\|_1\stackrel{\eqref{eq:cuts sparsified}}{\le} (1+\e)\sum_{E\subseteq \{1,\ldots,n\}} w_E |\1_E(i)-\1_E(j)|
\stackrel{\eqref{eq:cut cone}}{=}\|f_i-f_j\|_1.\qedhere
\end{multline*}
\end{proof}

\begin{rema}
Talagrand~\cite{Ta1} proved that any $n$-dimensional linear subspace of $L_1$ embeds with distortion $1+\e$ into $\ell_1^k$, with $k\le C(\e)n\log n$. This strengthens Schechtman's bound in~\eqref{eq:cases} for $n$-point subsets of $L_1$, since it achieves a low dimensional embedding of their span. It would be very interesting to remove the $\log n$ term in Talagrand's theorem, as this would clearly be best possible. Note that $n$-point subsets of $L_1$ can conceivably be embedded into $\ell_1^k$, with $k\ll n$. Embedding into at least $n$ dimensions (with any finite distortion) is a barrier whenever the embedding proceeds by actually embedding the span of the given $n$ points. The Newman-Rabinovich argument based on sparsification proceeds differently, and one might hope that it could be used to break the $n$ dimensions  barrier for $n$-point subsets of $L_1$. This turns out to be possible:  the forthcoming paper~\cite{ANN} shows that for any $D>1$,  any $n$-point  subset of $L_1$ embeds with distortion $D$ into $\ell_1^k$, with $k=O(n/D)$.
\end{rema}

\subsection{Finite dimensional subspaces of $L_p$ for even $p$}\label{sec:Lp}

Given an $n\in \N$ and $\e\in (0,1)$, what is the smallest $k\in \N$ such that any $n$-dimensional subspace of $L_p$ linearly embeds with distortion $1+\e$ into $\ell_p^k$? This problem has been studied extensively~\cite{Sc1,Sc2,BLM,Ta1,Ta2,JS1,SZ,Zv,JS2}, the best known bound on $k$ being as follows.
\begin{equation*}\label{eq:case2}
k\le \left\{
\begin{array}{ll}
C(p,\e)n\log n(\log\log n)^2 & p\in (0,1) \quad\ \, \text{\cite{Zv}},\\
C(\e)n\log n & p=1 \quad\quad\quad\, \text{\cite{Ta1}},\\
C(\e)n\log n(\log\log n)^2 & p\in (1,2)\quad\ \, \text{\cite{Ta2}},\\
C(p,\e)n^{p/2}\log n & p\in (2,\infty) \quad \text{\cite{BLM}}.
\end{array}
\right.
\end{equation*}
In particular, Bourgain, Lindenstrauss and Milman~\cite{BLM} proved that if $p\in (2,\infty)$ then one can take $k\le C(p,\e)n^{p/2}\log n$. It was long known~\cite{BDGJN}, by considering subspaces of $L_p$ that are almost isometric to $\ell_2^n$, that necessarily $k\ge c(p,\e)n^{p/2}$. We will now show an elegant argument of Schechtman, based on Theorem~\ref{thm:psd version}, that removes the $\log n$ factor when $p$ is an even integer, thus obtaining the first known sharp results for some values of $p\neq 2$.

\begin{theo}\label{thm:schechtman}
Assume that $p>2$ is an even integer,  $n\in \N$ and $\e\in (0,1)$. Then any $n$-dimensional subspace $X$ of $L_p$ embeds with distortion $1+\e$ into $\ell_p^k$ for some $k\le (cn/p)^{p/2}/\e^2$, where $c$ is a universal constant.
\end{theo}

\begin{proof}
By a standard argument (approximating a net in the sphere of $X$ by simple functions), we may assume that $X\subseteq \ell_p^m$ for some finite (huge) $m\in \N$. In what follows, when we use multiplicative notation for vectors in $\R^m$, we mean coordinate-wise products, i.e., for $x,y\in \R^m$, write $xy=(x_1y_1,\ldots,x_my_m)$ and for $r\in \N$ write $x^r=(x_1^r,\ldots,x_m^r)$.

Let $u_1,\ldots, u_n$ be a basis of $X$. Consider the following subspace of $\R^m$:
$$
Y=\mathrm{span}\left(\left\{u_{j_1}^{p_1}u_{j_2}^{p_2}\cdots u_{j_\ell}^{p_\ell}:\ \ell\in \N,\ j_1,\ldots,j_\ell\in \{1,\ldots,n\},\ p_1+\cdots+p_\ell=\frac{p}{2}\right\}\right).
$$
Then
$$
d=\dim(Y)\le \binom{n+p/2-1}{p/2}\le \left(\frac{10n}{p}\right)^{p/2}.
$$
Thinking of $Y$ as a $d$-dimensional subspace of $\ell_2^m$, let $v_1,\ldots,v_d$ be an orthonormal basis of $Y$. Define $x_1,\ldots,x_m\in Y$ by $x_i=\sum_{j=1}^d \langle v_j,e_i\rangle v_j$, where as usual $e_1,\ldots,e_m$ is the standard coordinate basis of $\R^m$. Note that by definition (since $v_1,\ldots,v_d$ is an orthonormal basis of $Y$), for every $y\in Y$ and for every $i\in \{1,\ldots,m\}$ we have $\langle x_i,y\rangle=\langle y,e_i\rangle =y_i$. By Theorem~\ref{thm:psd version}, there exists a subset $\sigma\subseteq \{1,\ldots,m\}$ with $|\sigma|=O(d/(p\e)^2)\le (cn/p)^{p/2}/\e^2$, and $\{s_i\}_{i\in \sigma}\subseteq (0,\infty)$, such that for all $y\in Y$ we have
\begin{equation}\label{eq:gid}
\sum_{i=1}^m y_i^2\le \sum_{i\in \sigma} s_iy_i^2\le \left(1+\frac{\e p}{4}\right) \sum_{i=1}^m y_i^2.
\end{equation}
In particular, since by the definition of $Y$ for every $x\in X$ we have $x^{p/2}\in Y$,
$$
\|x\|_p=\left(\sum_{i=1}^m x_i^p\right)^{1/p}\stackrel{\eqref{eq:gid}}{\le} \left(\sum_{i\in \sigma} s_ix_i^p\right)^{1/p}\stackrel{\eqref{eq:gid}}{\le} \left(1+\frac{\e p}{4}\right)^{1/p}\left(\sum_{i=1}^m x_i^p\right)^{1/p}\le (1+\e)\|x\|_p.
$$
Thus  $x\mapsto (s_i^{1/p}x_i)_{i\in \sigma}$ maps $X$ into $\ell_p^\sigma\subseteq \ell_p^{(cn/p)^{p/2}/\e^2}$ and has distortion $1+\e$.
\end{proof}

\begin{rema}
The bound on $k$ in Theorem~\ref{thm:schechtman} is sharp also in terms of the dependence on $p$. See~\cite{Sc3} for more information on this topic.
\end{rema}


\section{The restricted invertibility principle}\label{sec:BT}


In this section square matrices are no longer assumed to be symmetric. The ensuing discussion does not deal with a direct application of the statement of Theorem~\ref{thm:psd version}, but rather with an application of the method that was introduced by Batson-Spielman-Srivastava to prove Theorem~\ref{thm:psd version}.

Bourgain and Tzafriri studied in~\cite{BT1,BT2,BT3}  conditions on matrices which ensure that they have large ``well invertible" sub-matrices, where well invertibility refers to control of the operator norm of the inverse. Other than addressing  a fundamental question, such phenomena are very important to a variety of interesting applications that we will not survey here.

To state the main results of Bourgain-Tzafriri, we need the following notation. For $\sigma\subseteq \{1,\ldots,n\}$ let $R_\sigma:\R^n\to \R^\sigma$ be  given by restricting the coordinates to $\sigma$, i.e., $R_\sigma(\sum_{i=1}^n a_ie_i)=\sum_{i\in \sigma} a_ie_i$ (as usual, $\{e_i\}_{i=1}^n$ is the standard coordinate basis of $\R^n$). In matrix notation, given an operator $T:\R^n\to \R^n$, the operator $R_\sigma TR^*_\sigma:\R^\sigma\to \R^\sigma$ corresponds to the $\sigma\times \sigma$ sub-matrix $(\langle Te_i,e_j\rangle)_{i,j\in \sigma}$. The operator norm of $T$ (as an operator from $\ell_2^n$ to $\ell_2^n$) will be denoted below by $\|T\|$, and the  Hilbert-Schmidt norm of $T$ will be denoted $\|T\|_{\HS}=\sqrt{\sum_{i=1}^n\sum_{j=1}^n \langle Te_i,e_j\rangle^2}$.

The following theorem from~\cite{BT1,BT3} is known as the Bourgain-Tzafriri restricted invertibility principle.

\begin{theo}\label{thm:BT1} There exist universal constants $c,K>0$ such that for every $n\in \N$ and every linear operator $T:\R^n\to \R^n$ the following assertions hold true:
\begin{enumerate}
\item If $\|Te_i\|_2=1$ for all $i\in \{1,\ldots,n\}$ then there exists a subset
$\sigma\subseteq \{1,\ldots,n\}$ satisfying
\begin{equation}\label{eq:old sigma}
|\sigma|\ge  \frac{c n}{\|T\|^2},
\end{equation}
such that $R_\sigma T^*T R_\sigma^*$ is invertible and
\begin{equation}\label{eq:K}
\left\|\left(R_\sigma T^*T R_\sigma^*\right)^{-1}\right\|\le K.
\end{equation}
\item If $\langle Te_i,e_i\rangle=1$ for all $i\in \{1,\ldots,n\}$ then for all $\e\in (0,1)$ there exists a subset $\sigma\subseteq \{1,\ldots,n\}$ satisfying
\begin{equation}\label{eq:quadratic}
|\sigma|\ge \frac{c\e^2 n}{\|T\|^2},
\end{equation}
such that $R_\sigma T^*T R_\sigma^*$ is invertible and
\begin{equation}\label{eq:1+eps}
\left\|\left(R_\sigma T^*T R_\sigma^*\right)^{-1}\right\|\le 1+\e.
\end{equation}
\end{enumerate}
\end{theo}
The quadratic dependence on $\e$ in~\eqref{eq:quadratic} cannot be
improved~\cite{BHKW}. Observe that~\eqref{eq:K} is equivalent to the
following assertion:
\begin{equation}\label{BT norm condition}
\left\|\sum_{i\in \sigma} a_i Te_i\right\|_2^2\ge
\frac{1}{K}\sum_{i\in \sigma} a_i^2\quad\quad \forall \{a_i\}_{i\in
\sigma}\subseteq \R.
\end{equation}
We note that if $T$ satisfies the assumption of the first assertion
of Theorem~\ref{thm:BT1} then $T^*T$ satisfies the assumption of the
second assertion of Theorem~\ref{thm:BT1}. Hence, the second
assertion of Theorem~\ref{thm:BT1} implies the first assertion of
Theorem~\ref{thm:BT1} with~\eqref{eq:K} replaced by
$\left\|\left(R_\sigma T^*T R_\sigma^*\right)^{-1}\right\|\le
(1+\e)\|T\|^2$ and~\eqref{eq:old sigma} replaced by the condition
$|\sigma|\ge c\e^2n/\|T\|^4$.


In~\cite{SS2} Spielman and Srivastava proved the following theorem:
\begin{theo}\label{thm:SS RI}
Suppose that $x_1,\ldots,x_m\in \R^n\setminus\{0\}$ satisfy
\begin{equation}\label{eq:RI decomposition}
\sum_{i=1}^m x_i\otimes x_i=I.
\end{equation}
Then for every linear $T:\R^n\to \R^n$ and $\e\in (0,1)$ there exists $\sigma\subseteq \{1,\ldots,m\}$ with
\begin{equation}\label{eq:RI thm quadratic}
|\sigma|\ge \left\lfloor\frac{\e^2\|T\|_{\HS}^2}{\|T\|^2}\right\rfloor,
\end{equation}
and such that for all $\{a_i\}_{i\in \sigma}\subseteq \R$ we have
\begin{equation}\label{eq:m}
\left\|\sum_{i\in \sigma} a_iTx_i\right\|_2^2\ge \frac{(1-\e)^2\|T\|_{\HS}^2}{m}\sum_{i\in \sigma}a_i^2.
\end{equation}
\end{theo}
Theorem~\ref{thm:SS RI} implies the Bourgain-Tzafriri restricted
invertibility principle. Indeed,  take $x_i=e_i$ and note that if
either $\|Te_i\|_2=1$ for all $i\in \{1,\ldots,n\}$ or $\langle
Te_i,e_i\rangle=1$ for all $i\in \{1,\ldots,n\}$ then
$\|T\|_{\HS}^2\ge n$. The idea to improve the Bourgain-Tzafriri
theorem in terms of Hilbert-Schmidt estimates is due to Vershynin,
who proved in~\cite{Ve} a statement similar to Theorem~\ref{thm:SS
RI} (with asymptotically worse dependence on $\e$). Among the tools
used in Vershynin's argument  is the Bourgain-Tzafriri restricted
invertibility theorem itself, but we will see how the iterative
approach of Section~\ref{sec:proof} yields a self-contained and
quite simple proof of Theorem~\ref{thm:SS RI}. This new approach of
Spielman-Srivastava has other advantages. Over the years, there was
interest~\cite{BT1,BT3,Tr1,CT} in improving the quantitative
estimates in Theorem~\ref{thm:BT1} (i.e., the bounds on $c, K$, and
the dependence $|\sigma|$ on $\e$ and $\|T\|$), and
Theorem~\ref{thm:SS RI} yields the best known bounds. Moreover, it
is not obvious that the subset $\sigma$ of Theorem~\ref{thm:BT1} can
be found in polynomial time. A randomized algorithm achieving this
was recently found by Tropp~\cite{Tr1}, and the work of
Spielman-Srivastava yields a determinstic algorithm which finds in
polynomial time a subset $\sigma$ satisfying the assertions of
Theorem~\ref{thm:SS RI}.

Before proceeding to an exposition of the proof of
Theorem~\ref{thm:SS RI} in~\cite{SS2}, we wish to note that another
important result of Bourgain-Tzafriri~\cite{BT1,BT2} is the
following theorem, which is easily seen to imply the second
assertion of Theorem~\ref{thm:BT1} with the
conclusion~\eqref{eq:1+eps} replaced by $\left\|\left(R_\sigma T
R_\sigma^*\right)^{-1}\right\|\le 1+\e$. This theorem is important
for certain applications, and it would be interesting if it could be
proved using the Spielman-Srivastava method as well.
\begin{theo}\label{thm:BT0}
There is a universal constant $c>0$ such that for every $\e>0$ and
$n\in \N$ if an operator $T:\R^n\to \R^n$ satisfies $\langle
Te_i,e_i\rangle=0$ for all $i\in \{1,\ldots,n\}$ then there exists a
subset $\sigma\subseteq \{1,\ldots, n\}$ satisfying $|\sigma|\ge
c\e^2 n$ and $\|R_\sigma TR_\sigma^*\|\le \e\|T\|.$
\end{theo}

\subsection{Proof of Theorem~\ref{thm:SS RI}}
The conclusion~\eqref{eq:m} of Theorem~\ref{thm:SS RI} is equivalent to the requirement that the matrix
\begin{equation}\label{eq:def A}
A=\sum_{i\in \sigma} (Tx_i)\otimes (Tx_i)
 \end{equation}
 has $|\sigma|$ eigenvalues at least $(1-\e)^2\|T\|_{\HS}^2/m$. Indeed, if $B$ is the $|\sigma|\times n$ matrix whose rows are $\{Tx_i\}_{i\in\sigma}$, then $A=B^*B$. The eigenvalues of $A$ are therefore the same as the eigenvalues of the $|\sigma|\times |\sigma|$ Gram matrix $BB^*=(\langle Tx_i,Tx_j\rangle)_{i,j\in \sigma}$. The assertion that all the eigenvalues of $BB^*$ are at least $(1-\e)^2\|T\|_{\HS}^2/m$ is identical to~\eqref{eq:m}.


Define
\begin{equation}\label{eq:def k}
k=\left\lfloor\frac{\e^2\|T\|_{\HS}^2}{\|T\|^2}\right\rfloor.
 \end{equation}
 We will construct inductively $y_0,y_1,\ldots,y_k\in \R^n$ with the following properties. We set $y_0=0$ and require that $y_1,\ldots,y_k\in \{x_1,\ldots,x_m\}$. Moreover, if for $i\in \{0,\ldots,k\}$ we write
\begin{equation}\label{eq:def b_i}
b_i=\frac{(1-\e)}{m}\left(\|T\|_{\HS}^2-\frac{i}{\e}\|T\|^2\right),
\end{equation}
then the  matrix
\begin{equation}\label{eq:def A_i}
A_i=\sum_{j=0}^i (Ty_j)\otimes (Ty_j)
\end{equation}
has  $k$ eigenvalues bigger than $b_i$ and all its other eigenvalues equal $0$ (this holds vacuously for $i=0$). Note that this requirement implies in particular that $y_1,\ldots,y_k$ are distinct. Finally, we require that for every $i\in \{1,\ldots,k\}$ we have
    \begin{equation}\label{eq:new barrier}
    \sum_{j=1}^m \left\langle \left(A_i-b_iI\right)^{-1}Tx_j,Tx_j\right\rangle
     <\sum_{j=1}^m \left\langle \left(A_{i-1}-b_{i-1}I\right)^{-1}Tx_j,Tx_j\right\rangle.
    \end{equation}
The matrix $A_k$ will then have the form~\eqref{eq:def A} with $|\sigma|=k$, and have  $k$ eigenvalues greater than $(1-\e)^2\|T\|_{\HS}^2/m$, as required. It remains therefore to show that for $i\in \{1,\ldots, k\}$ there exists a vector $y_i$ satisfying the desired properties, assuming that $y_0,y_1,\ldots,y_{i-1}$ have already been selected.

\begin{lemm}\label{lem:summed second}
Denote
\begin{equation}\label{eq:def mu}
\mu=\sum_{j=1}^m \left\langle \left(A_{i-1}-b_{i-1}I\right)^{-1}Tx_j,Tx_j\right\rangle-\sum_{j=1}^m \left\langle \left(A_{i-1}-b_iI\right)^{-1}Tx_j,Tx_j\right\rangle.
\end{equation}
(Since $b_i\in (0,b_{i-1})$, the matrix $(A_{i-1}-b_iI)^{-1}$ makes sense in~\eqref{eq:def mu}.) Then
\begin{multline}\label{eq:summed ineq second}
\sum_{j=1}^m \left\langle \left(A_{i-1}-b_{i}I\right)^{-1}TT^*\left(A_{i-1}-b_{i}I\right)^{-1}Tx_j,Tx_j\right\rangle\\< -\mu\sum_{j=1}^m \left(1+\left\langle \left(A_{i-1}-b_iI\right)^{-1}Tx_j,Tx_j\right\rangle\right).
\end{multline}
\end{lemm}
Assuming the validity of Lemma~\ref{lem:summed second} for the moment, we will show how to complete the inductive construction. By~\eqref{eq:summed ineq second} there exists $j\in \{1,\ldots, m\}$ satisfying
\begin{multline}\label{eq:j choice}
\left\langle \left(A_{i-1}-b_{i}I\right)^{-1}TT^*\left(A_{i-1}-b_{i}I\right)^{-1}Tx_j,Tx_j\right\rangle\\<-\mu\left(1+\left\langle \left(A_{i-1}-b_iI\right)^{-1}Tx_j,Tx_j\right\rangle\right).
\end{multline}
Our inductive choice will be $y_i=x_j$.

The matrix $\left(A_{i-1}-b_{i-1}I\right)^{-1}-\left(A_{i-1}-b_iI\right)^{-1}$ is positive definite, since by the inductive hypothesis its eigenvalues are all of the form $\left(\lambda-b_{i-1}\right)^{-1}-\left(\lambda-b_i\right)^{-1}$ for some $\lambda\in \R$ that satisfies $\lambda>b_{i-1}>b_i$ or $\lambda=0$ (and since for $i\le k$ we have $b_i>0$). Hence $\mu\ge 0$. Since the left hand side of~\eqref{eq:j choice} is nonnegative, it follows that
\begin{equation}\label{eq:1+ is neg}
1+\left\langle \left(A_{i-1}-b_iI\right)^{-1}Tx_j,Tx_j\right\rangle<0.
\end{equation}
Since $A_i=A_{i-1}+ (Tx_j)\otimes (Tx_j)$, it follows from~\eqref{eq:trace} that
\begin{equation}\label{eq:diff neg}
\trace\left(\left(A_{i}-b_iI\right)^{-1}\right)-\trace\left(\left(A_{i-1}-b_iI\right)^{-1}\right)=
-\frac{\left\langle \left(A_{i-1}-b_iI\right)^{-2}Tx_j,Tx_j\right\rangle}{1+\left\langle \left(A_{i-1}-b_iI\right)^{-1}Tx_j,Tx_j\right\rangle}\stackrel{\eqref{eq:1+ is neg}}{>}0,
\end{equation}
where in the last inequality of~\eqref{eq:diff neg} we used the fact that $\left(A_{i-1}-b_iI\right)^{-2}$ is positive definite. At the same time, by the inductive hypothesis $\lambda_1(A_{i-1}),\ldots,\lambda_{i-1}(A_{i-1})>b_{i-1}$, and $\lambda_{i}(A_{i-1})=\cdots=\lambda_n(A_{i-1})=0$. Since $A_i-A_{i-1}$ is a rank one positive semidefinite matrix, the eigenvalues of $A_i$ and $A_{i-1}$ interlace (see~\cite[Section~III.2]{Bah}; this result goes back to~\cite{Wey}), and therefore
\begin{equation}\label{eq:interlace1}
\lambda_1(A_i)\ge \lambda_1(A_{i-1})\ge\lambda_2(A_i)\ge \lambda_2(A_{i-1})\ge \cdots \ge \lambda_{i-1}(A_{i-1})\ge \lambda_{i}(A_i),
\end{equation}
and
\begin{equation}\label{eq:interlace2}
\lambda_{i}(A_{i-1})=\cdots=\lambda_n(A_{i-1})=\lambda_{i+1}(A_i)=\cdots=\lambda_n(A_{i})=0.
\end{equation}
Hence,
\begin{multline*}
0\stackrel{\eqref{eq:diff neg}}{<}\trace\left(\left(A_{i}-b_iI\right)^{-1}\right)-\trace\left(\left(A_{i-1}-b_iI\right)^{-1}\right)\\\stackrel{\eqref{eq:interlace2}}{=}
\frac{1}{\lambda_i(A_{i})-b_i}+\frac{1}{b_i}+\sum_{j=1}^{i-1}\left(\frac{1}{\lambda_j(A_i)-b_i}-\frac{1}{\lambda_j(A_{i-1})-b_i}\right)\stackrel{\eqref{eq:interlace1}}{\le}
\frac{\lambda_i(A_{i})}{b_i(\lambda_i(A_{i})-b_i)},
\end{multline*}
implying that $\lambda_i(A_i)>b_i$.

Therefore, in order to establish the inductive step, it remains to prove~\eqref{eq:new barrier}. To this end, note that due to~\eqref{eq:RI decomposition} and~\eqref{eq:trace identity} for every $n\times n$  matrix $A$ we have
\begin{equation}\label{eq:new trace identity}
\sum_{j=1}^m \left\langle ATx_j,Tx_j\right\rangle=\trace\left(T^*AT\right).
\end{equation}
Hence~\eqref{eq:new barrier} is equivalent to the inequality
\begin{equation}\label{eq:trace barrier}
\trace\left(T^*\left(A_{i-1}-b_{i-1}I\right)^{-1}T\right)>\trace\left(T^*\left(A_{i}-b_{i}I\right)^{-1}T\right).
\end{equation}
Now,
\begin{eqnarray*}
&&\!\!\!\!\!\!\!\!\!\!\!\!\!\!\!\!\!\!\!\!\!\!\!\!\!\!\!\!\!\!\!\!\!\!\!\!\!\!\!
\tr\left(T^*\left(A_{i}-b_{i}I\right)^{-1}T\right)-\tr\left(T^*\left(A_{i-1}-b_{i}I\right)^{-1}T\right)\\
&\stackrel{\eqref{eq:SM}}{=}&-
\frac{\tr\left(T^*\left(A_{i-1}-b_iI\right)^{-1}((Tx_j)\otimes (Tx_j))\left(A_{i-1}-b_iI\right)^{-1}T\right)}{1+\left\langle \left(A_{i-1}-b_iI\right)^{-1}Tx_j,Tx_j\right\rangle}\\
&=& -\frac{\left\langle \left(A_{i-1}-b_{i}I\right)^{-1}TT^*\left(A_{i-1}-b_{i}I\right)^{-1}Tx_j,Tx_j\right\rangle}{1+\left\langle \left(A_{i-1}-b_iI\right)^{-1}Tx_j,Tx_j\right\rangle}\\
&\stackrel{\eqref{eq:1+ is neg}\wedge\eqref{eq:j choice}}{<}&\mu
\stackrel{\eqref{eq:def mu}\wedge\eqref{eq:new trace identity}}{=} \tr\left(T^*\left(A_{i-1}-b_{i-1}I\right)^{-1}T\right)-\tr\left(T^*\left(A_{i-1}-b_{i}I\right)^{-1}T\right).
\end{eqnarray*}
This proves~\eqref{eq:trace barrier}, so all that remains in order to prove Theorem~\ref{thm:SS RI} is to prove Lemma~\ref{lem:summed second}.

\begin{proof}[Proof of Lemma~\ref{lem:summed second}] Using~\eqref{eq:new trace identity} we see that our goal~\eqref{eq:summed ineq second} is equivalent to the following inequality
\begin{multline}\label{eq:lem goal1}
\tr\left(T^*\left(A_{i-1}-b_{i}I\right)^{-1}TT^*\left(A_{i-1}-b_{i}I\right)^{-1}T\right)
\\< -\mu\left(m+\tr\left(T^*\left(A_{i-1}-b_{i}I\right)^{-1}T\right)\right).
\end{multline}
Note that
\begin{multline}\label{eq:take norm out}
\trace\left(T^*\left(A_{i-1}-b_{i}I\right)^{-1}TT^*\left(A_{i-1}-b_{i}I\right)^{-1}T\right)\\ \le \|T\|^2\trace\left(\left(A_{i-1}-b_{i}I\right)^{-1}TT^*\left(A_{i-1}-b_{i}I\right)^{-1}\right)=\|T\|^2\trace\left(T^*\left(A_{i-1}-b_{i}I\right)^{-2}T\right).
\end{multline}
The inductive hypothesis~\eqref{eq:new barrier}, or its equivalent form~\eqref{eq:trace barrier}, implies that
\begin{multline}\label{eq:before mu}
\trace\left(T^*\left(A_{i-1}-b_{i-1}I\right)^{-1}T\right)<\trace\left(T^*\left(A_{0}-b_{0}I\right)^{-1}T\right)
\\=-\frac{1}{b_0}\trace(T^*T)
=\frac{\|T\|_{\HS}^2}{b_0}\stackrel{\eqref{eq:def b_i}}{=}-\frac{m}{1-\e}.
\end{multline}
Hence,
\begin{equation}\label{eq:-mu bound}
\trace\left(T^*\left(A_{i-1}-b_{i}I\right)^{-1}T\right)\stackrel{\eqref{eq:def mu}\wedge\eqref{eq:new trace identity}\wedge\eqref{eq:before mu}}{<}-\frac{m}{1-\e}-\mu.
\end{equation}
From~\eqref{eq:take norm out} and~\eqref{eq:-mu bound} we see that in order to prove~\eqref{eq:lem goal1} it suffices to establish the following inequality:
\begin{equation}\label{eq:lem goal2}
\|T\|^2\trace\left(T^*\left(A_{i-1}-b_{i}I\right)^{-2}T\right)\le \frac{\e m}{1-\e}\mu+\mu^2.
\end{equation}

To prove~\eqref{eq:lem goal2} we first make some preparatory remarks. For $r\in \{0,\ldots,i-1\}$ let $P_r$ be the orthogonal projection on the image of $A_r$ and let $Q_r=I-P_r$ be the orthogonal projection on the kernel of $A_r$. Since $A_0=0$ we have $Q_0=I$. Moreover, because  $A_r=A_{r-1}+(Ty_r)\otimes (Ty_r)$ and $A_{r-1}, (Ty_r)\otimes (Ty_r)$ are both positive semidefinite, it follows that $\Ker(A_r)=\Ker(A_{r-1})\cap (Tx_r)^{\perp}$. Therefore
\begin{equation}\label{eq:trace recursion}
\tr(Q_{r-1}-Q_r)=\dim(\Ker(A_{r-1}))-\dim(\Ker(A_r))\le 1.
\end{equation}
Hence,
\begin{multline}\label{eq:HS recurse}
\|Q_rT\|_{\HS}^2=\trace\left(T^*Q_rT\right)=\|Q_{r-1}T\|_{\HS}^2-\tr\left(T^*(Q_{r-1}-Q_r)T\right)\\\ge
\|Q_{r-1}T\|_{\HS}^2-\|T\|^2\tr(Q_{r-1}-Q_r)\stackrel{\eqref{eq:trace recursion}}{\ge} \|Q_{r-1}T\|_{\HS}^2-\|T\|^2.
\end{multline}
Since $Q_0=I$, \eqref{eq:HS recurse} yields by induction the following useful bound:
\begin{equation}\label{eq:useful}
\|Q_{i-1}T\|_{\HS}^2\ge \|T\|_{\HS}^2-(i-1)\|T\|^2.
\end{equation}

Next, since the nonzero eigenvalues of $A_{i-1}$ are greater than $b_{i-1}$, the matrix $T^*P_{i-1}\left((A_{i-1}-b_{i-1}I)\left(A_{i-1}-b_iI\right)^{-2}\right)P_{i-1}T$ is positive semidefinite. In particular, its trace is nonnegative, yielding the following estimate:
\begin{multline*}
0\le \trace\left(T^*P_{i-1}\left((A_{i-1}-b_{i-1}I)\left(A_{i-1}-b_iI\right)^{-2}\right)P_{i-1}T\right)\\
= \trace\left(T^*P_{i-1}\left(\frac{(A_{i-1}-b_{i-1}I)^{-1}-(A_i-b_iI)^{-1}}{(b_{i-1}-b_i)^2}-\frac{(A_{i-1}-b_iI)^{-2}}{b_{i-1}-b_i}\right)P_{i-1}T\right),
\end{multline*}
which rearranges to the following inequality:
\begin{multline}\label{eq:incrementP}
(b_{i-1}-b_i)\trace\left(T^*P_{i-1}\left(A_{i-1}-b_iI\right)^{-2}P_{i-1}T\right)\\
\le \trace\left(T^*P_{i-1}\left(A_{i-1}-b_{i-1}I\right)^{-1}P_{i-1}T\right)-\trace\left(T^*P_{i-1}\left(A_{i-1}-b_{i}I\right)^{-1}P_{i-1}T\right).
\end{multline}
Since $Q_{i-1}(A_{i-1}-b_{i-1}I)^{-1}Q_{i-1}=-\frac{1}{b_{i-1}}Q_{i-1}$ and $Q_{i-1}(A_{i-1}-b_{i}I)^{-1}Q_{i-1}=-\frac{1}{b_{i}}Q_{i-1}$,
\begin{eqnarray}\label{eq:mu lower}
\mu&=&\nonumber\trace\left(T^*(P_{i-1}+Q_{i-1})\left(A_{i-1}-b_{i-1}I\right)^{-1}(P_{i-1}+Q_{i-1})T\right)\\\nonumber&&
-\trace\left(T^*(P_{i-1}+Q_{i-1})\left(A_{i-1}-b_{i}I\right)^{-1}(P_{i-1}+Q_{i-1})T\right)
\\&=&\nonumber\trace\left(T^*P_{i-1}\left(A_{i-1}-b_{i-1}I\right)^{-1}P_{i-1}T\right)-\trace\left(T^*P_{i-1}\left(A_{i-1}-b_{i}I\right)^{-1}P_{i-1}T\right)\\
&&+\nonumber \left(\frac{1}{b_i}-\frac{1}{b_{i-1}}\right)\trace\left(T^*Q_{i-1}T\right)\\
&\stackrel{\eqref{eq:incrementP}}{\ge} & (b_{i-1}-b_i)\trace\left(T^*P_{i-1}\left(A_{i-1}-b_iI\right)^{-2}P_{i-1}T\right)+\frac{b_{i-1}-b_i}{b_{i-1}b_i}\|Q_{i-1}T\|_{\HS}^2.
\end{eqnarray}
Also $Q_{i-1}(A_{i-1}-b_{i}I)^{-2}Q_{i-1}=\frac{1}{b_{i}^2}Q_{i-1}$, and therefore
\begin{eqnarray*}
\trace\left(T^*\left(A_{i-1}-b_{i}I\right)^{-2}T\right)&=&\trace\left(T^*P_{i-1}\left(A_{i-1}-b_iI\right)^{-2}P_{i-1}T\right)+\frac{\trace(T^*Q_{i-1}T)}{b_i^2}\\
&=&\trace\left(T^*P_{i-1}\left(A_{i-1}-b_iI\right)^{-2}P_{i-1}T\right)+\frac{1}{b_i^2}\|Q_{i-1}T\|_{\HS}^2\\
&\stackrel{\eqref{eq:mu lower}}{\le}& \frac{\mu}{b_{i-1}-b_i}+\frac{\|Q_{i-1}T\|_{\HS}^2}{b_i}\left(\frac{1}{b_i}-\frac{1}{b_{i-1}}\right)\\
&\stackrel{\eqref{eq:def b_i}}{=}&\frac{\e m\mu}{(1-\e)\|T\|^2}+\frac{\|Q_{i-1}T\|_{\HS}^2}{b_i}\left(\frac{1}{b_i}-\frac{1}{b_{i-1}}\right).
\end{eqnarray*}
It follows that in order to prove the desired inequality~\eqref{eq:lem goal2}, it suffices to show that the following inequality holds true:
\begin{equation}\label{eq:goal 3}
\|T\|^2\frac{\|Q_{i-1}T\|_{\HS}^2}{b_i}\left(\frac{1}{b_i}-\frac{1}{b_{i-1}}\right)\le \mu^2.
\end{equation}
Since $b_{i-1}>b_i$ and $T^*P_{i-1}\left(A_{i-1}-b_iI\right)^{-2}P_{i-1}T$ is positive semidefinite, a consequence of~\eqref{eq:mu lower} is that $\mu\ge \|Q_{i-1}T\|_{\HS}^2\left(\frac{1}{b_i}-\frac{1}{b_{i-1}}\right)$. Hence, in order to prove~\eqref{eq:goal 3} it suffices to show that
$$
\|T\|^2\frac{\|Q_{i-1}T\|_{\HS}^2}{b_i}\left(\frac{1}{b_i}-\frac{1}{b_{i-1}}\right)\le \|Q_{i-1}T\|_{\HS}^4\left(\frac{1}{b_i}-\frac{1}{b_{i-1}}\right)^2,
$$
or equivalently,
$$
\|Q_{i-1}T\|_{\HS}^2\ge \|T\|^2\frac{b_{i-1}}{b_{i-1}-b_i}\stackrel{\eqref{eq:def b_i}}{=}\e\|T\|_{\HS}^2-(i-1)\|T\|^2,
$$
which is a consequence of inequality~\eqref{eq:useful}, that we proved earlier.
\end{proof}

\section{Nonlinear notions of sparsification}\label{sec:other}

Quadratic forms such as $\sum_{i=1}^n\sum_{j=1}^n g_{ij}(x_i-x_j)^2$ are expressed in terms of the mutual distances between the points $\{x_1,\ldots,x_n\}\subseteq \R$. This feature makes them very useful for a variety of applications in metric geometry, where the Euclidean distance is replaced by other geometries. We refer to~\cite{MN1,NS} for a (partial) discussion of such issues. It would be useful to study the sparsification problem of Theorem~\ref{thm:BSS} in the non-Euclidean setting as well, although the spectral arguments used by Batson-Spielman-Srivastava seem inadequate for addressing such nonlinear questions.

In greatest generality one might consider an abstract set $X$, and a symmetric function (kernel) $K:X\times X\to [0,\infty)$. Given an $n\times n$ matrix $G=(g_{ij})$, the goal would be to find a sparse $n\times n$ matrix $H=(h_{ij})$ satisfying
\begin{equation}\label{eq:kernel}
\sum_{i=1}^n\sum_{j=1}^n g_{ij} K(x_i,x_j)\le \sum_{i=1}^n\sum_{j=1}^n h_{ij} K(x_i,x_j)\le C\sum_{i=1}^n\sum_{j=1}^n g_{ij} K(x_i,x_j),
\end{equation}
for some constant $C>0$ and all $x_1,\ldots,x_n\in X$.

Cases of geometric interest in~\eqref{eq:kernel} are when $K(x,y)=d(x,y)^p$, where $d(\cdot,\cdot)$ is a metric on $X$ and $p>0$. When $p\neq 2$ even the case of the real line with the standard metric is unclear. Say that an $n\times n$ matrix $H=(h_{ij})$ is a $p$-sparsifier with quality $C$ of an $n\times n$ matrix $G=(g_{ij})$ if $\supp(H)\subseteq \supp(G)$ and there exists a scaling factor $\lambda>0$ such that for every $x_1,\ldots,x_n\in \R$ we have
\begin{equation}\label{eq:power p}
\lambda\sum_{i=1}^n\sum_{j=1}^n g_{ij} |x_i-x_j|^p\le \sum_{i=1}^n\sum_{j=1}^n h_{ij} |x_i-x_j|^p\le
C\lambda\sum_{i=1}^n\sum_{j=1}^n g_{ij} |x_i-x_j|^p.
\end{equation}
By integrating~\eqref{eq:power p} we see that it is equivalent to the requirement that for every $f_1,\ldots,f_n\in L_p$ we have
\begin{equation}\label{eq:L-p version}
\lambda\sum_{i=1}^n\sum_{j=1}^n g_{ij} \|f_i-f_j\|_p^p\le \sum_{i=1}^n\sum_{j=1}^n h_{ij} \|f_i-f_j\|_p^p\le
C\lambda\sum_{i=1}^n\sum_{j=1}^n g_{ij} \|f_i-f_j\|_p^p.
\end{equation}
By a classical theorem of Schoenberg~\cite{Scho} (see also~\cite{WW}), if $q\le p$ then the metric space $(\R,|x-y|^{q/p})$ admits an isometric embedding into $L_2$, which in turn is isometric to a subspace of $L_p$. It therefore follows from~\eqref{eq:L-p version} that if $H$ is a $p$-sparsifier of $G$ with quality $C$ then it is also a $q$-sparsifier of $G$ with quality $C$ for every $q\le p$. In particular, when $p\in (0,2)$, Theorem~\ref{thm:BSS} implies that for every $G$ a $p$-sparsifier $H$ of quality $1+\e$ always exists with $|\supp(H)|=O(n/\e^2)$.

When $p>2$ it is open whether every matrix $G$ has a good $p$-sparsifier $H$. By ``good" we mean that the quality of the sparsifier $H$ is small, and that $|\supp(H)|$ is small. In particular, we ask whether every matrix $G$ admits a $p$-sparisfiers $H$ with quality $O_p(1)$ (maybe even $1+\e$) and $|\supp(H)|$ growing linearly with $n$.

It was shown to us by Bo'az Klartag that if $G=(g_ig_j)$ is a product matrix with nonnegative entries then Matou\v{s}ek's extrapolation argument for Poincar\'e inequalities~\cite{Ma}  (see also~\cite[Lemma~4.4]{NS}) can be used to show that if $q>p$ and  $H$ is a $p$-sparsifier of $G$ with quality $C$, then $H$ is also a $q$-sparsifier of $G$ with quality $C'(C,p,q)$. However, we shall now present a simple example showing that a $p$-sparsifier of $G$ need not be a $q$-sparsifier of $G$ with quality independent of $n$ for any $q>p$, for some matrix $G$ (which is, of course, not a product matrix). This raises the question whether or not the method of Batson-Spielman-Srivastava, i.e., Theorem~\ref{thm:BSS}, produces a matrix $H$ which is a $O(1)$-quality $p$-sparsifier of $G$ for some $p>2$.

Fix $q>p$, $\e>0$ and $n\in \N$. Let $G=(g_{ij})$ be the $n\times n$ adjacency matrix of the weighted $n$-cycle, where one edge has weight $1$, and all remaining edges have weight $(n-1)^{p-1}/\e$, i.e., $g_{1n}=g_{n1}=1$,
$$
g_{12}=g_{21}=g_{23}=g_{32}=\cdots=g_{n-1,n}=g_{n,n-1}=\frac{(n-1)^{p-1}}{\e},
 $$
 and all the other entries of $G$ vanish. Let $H=(h_{ij})$ be the adjacency matrix of the same weighted graph, with the edge $\{1,n\}$ deleted, i.e., $h_{1n}=h_{n1}=0$ and all the other entries of $H$ coincide with the entries of $G$. It is immediate from the definition that $\sum_{i=1}^n\sum_{j=1}^n g_{ij}|x_i-x_j|^p\ge \sum_{i=1}^n\sum_{j=1}^n h_{ij}|x_i-x_j|^p$ for all $x_1,\ldots,x_n\in \R$. The reverse inequality is proved as follows:
\begin{eqnarray*}
\sum_{i=1}^n\sum_{j=1}^n g_{ij}|x_i-x_j|^p&=&2|x_1-x_n|^p+\frac{2(n-1)^{p-1}}{\e}\sum_{i=1}^{n-1}|x_i-x_{i+1}|^p\\
&\le&
2\left(\sum_{i=1}^{n-1}|x_i-x_{i+1}|\right)^p+\frac{2(n-1)^{p-1}}{\e}
\sum_{i=1}^{n-1}|x_i-x_{i+1}|^p\\
&\le & (1+\e)\frac{2(n-1)^{p-1}}{\e}
\sum_{i=1}^{n-1}|x_i-x_{i+1}|^p\\
&=& (1+\e)\sum_{i=1}^n\sum_{j=1}^n h_{ij}|x_i-x_j|^p.
\end{eqnarray*}
Hence $H$ is a $p$-sparsifier of $G$ with quality $1+\e$.

For the points $x_i=i$ we have
$
\sum_{i=1}^n\sum_{j=1}^n g_{ij}|x_i-x_j|^q=2(n-1)^q+2(n-1)^{p}/\e,
$
and
$
\sum_{i=1}^n\sum_{j=1}^n h_{ij}|x_i-x_j|^q=2(n-1)^{p}/\e.
$
At the same time, if $y_2=1$ and $y_i=0$ for all $i\in \{1,\ldots,n\}\setminus \{2\}$, we have
$\sum_{i=1}^n\sum_{j=1}^n g_{ij}|y_i-y_j|^q=\sum_{i=1}^n\sum_{j=1}^n h_{ij}|y_i-y_j|^q>0$. Thus, the quality of $H$ as a $q$-sparsifier of $G$ is at least $\e(n-1)^{q-p}$, which tends to $\infty$ with $n$, since $q>p$.

\bigskip


\noindent{\bf Acknowledgments.} This paper is a survey of recent
results of various authors, most notably
Batson-Spielman-Srivastava~\cite{BSS},
Spielman-Srivastava~\cite{SS1,SS2}, Srivastava~\cite{Sr1,Sr2},
Newman-Rabinovich~\cite{NR} and Schechtman~\cite{Sc3}. Any
differences between the presentation here and the results being
surveyed are only cosmetic. I am very grateful to Alexandr Andoni,
Tim Austin, Keith Ball, Bo'az Klartag, Ofer Neiman and especially
Vincent Lafforgue, Gilles Pisier, Gideon Schechtman and Nikhil
Srivastava, for helpful discussions and suggestions.


\end{document}